\documentclass[12pt]{article}
\usepackage[english]{babel}
\usepackage[latin1]{inputenc}
\usepackage[T1]{fontenc}

\usepackage{graphics,graphicx} 

\usepackage{latexsym}
\usepackage{amsfonts}
\usepackage{amsthm}
\usepackage{amsmath}
\usepackage{eepic,epic}
\usepackage{graphicx}
\usepackage[hmargin=2.3cm,vmargin=2.5cm]{geometry}

\newtheorem{theorem}{Theorem}

\newtheorem{proposition}[theorem]{Proposition}
\newtheorem{remark}[theorem]{Remark}
\newtheorem{problem}{Problem}

\title{On graphs with representation number 3}
\author{Sergey Kitaev\footnote{Department of Computer and Information Sciences, University of Strathclyde, 26 Richmond Street, Glasgow, G1 1XH, United Kingdom. {\bf Email:} sergey.kitaev@cis.strath.ac.uk.}}

\begin{document}

\maketitle
\thispagestyle{empty}

\begin{abstract} 
A graph $G=(V,E)$ is word-representable if there exists a word $w$ over the alphabet $V$ such that letters $x$ and $y$ alternate in $w$ if and only if $(x,y)$ is an edge in $E$. A graph is word-representable if and only if it is $k$-word-representable for some $k$, that is, if there exists a word containing $k$ copies of each letter that represents the graph. Also, being $k$-word-representable implies being $(k+1)$-word-representable. The minimum $k$ such that a word-representable graph is $k$-word-representable, is called graph's representation number. 

Graphs with representation number 1 are complete graphs, while graphs with representation number 2 are circle graphs. The only fact known before this paper on the class of graphs with representation number 3, denoted by $\mathcal{R}_3$, is that the Petersen graph and triangular prism belong to this class.  In this paper, we show that any prism belongs to $\mathcal{R}_3$, and that two particular operations  of extending graphs preserve the property of being in   $\mathcal{R}_3$. Further, we show that $\mathcal{R}_3$ is not included in a class of $c$-colorable graphs for a constant $c$. To this end, we extend  three known results related to operations on graphs. 

We also show that ladder graphs used in the study of prisms are $2$-word-representable, and thus each ladder graph is a circle graph.  Finally, we discuss $k$-word-representing comparability graphs via consideration of crown graphs, where we state some problems for further research.\\

\end{abstract}

\section{Introduction}\label{intro}

A graph $G=(V,E)$ is word-representable if there exists a word $w$ over the alphabet $V=V(G)$ such that letters $x$ and $y$ alternate in $w$ if and only if $(x,y)$ is an edge in $E=E(G)$. It follows from definitions that word-representable graphs are a {\em hereditary class} of graphs. A comprehensive introduction to the theory of word-representable graphs is given in~\cite{KL}.

A graph is word-representable if and only if it is $k$-word-representable for some $k$, that is, if there exists a word containing $k$ copies of each letter that represents the graph (see Theorem~\ref{repr-must-be-k-repr}). By Proposition~\ref{k-implies-k-plus-1}, being $k$-word-representable implies being $(k+1)$-word-representable. The minimum $k$ such that a word-representable graph is $k$-word-representable, is called {\em graph's representation number}. We let $\mathcal{R}_k$ denote the class of graphs having representation number $k$. 

$\mathcal{R}_1$ is easy to see to be the class of complete graphs (also known as cliques), while $\mathcal{R}_2$ is the class of circle graphs (see Theorem~\ref{2-repr-characterization}). The only fact known before this paper on $\mathcal{R}_3$ was that the Petersen graph (to the right in Figure~\ref{fig:representable-graphs}) and the triangular prism (to the left in Figure~\ref{prisms-pic}) belong to this class. Theorem~\ref{known-R3} extends our knowledge on $\mathcal{R}_3$, in particular, showing that {\em all} prisms belong to this class (also, see Theorem~\ref{Prn-repr-num-3-thm}).

In Section~\ref{glueing-connecting-two-graphs}, we revise  connecting two word-representable graphs, say $G_1$ and $G_2$, by an edge and gluing these graphs in a vertex, originally studied in \cite{KP2008}.  One can use a graph orientation argument involving Theorem~\ref{thm:rep-equals-semi-trans}, or results in \cite{KP2008}, 
to show that the resulting graph $G$, in both cases, is word-representable.   However, these do not answer directly the following question: If $G_1$ is $k_1$-word-representable, $G_2$ is $k_2$-word-representable, and $G$ is $k$-word-representable (such a $k$ must exist by Theorem~\ref{repr-must-be-k-repr}) then what can be said about $k$? Theorem~\ref{connect-glue-thm} answers the question.

Further,  in Section~\ref{modules-subsec} we revise replacing a vertex in a graph with a module, briefly considered in~\cite{HKP2010}. Theorem~\ref{module-thm} in that section is an extended version of an observation made in \cite{HKP2010}. In Section~\ref{sec-c-colorable-graphs}, providing two arguments based on Theorems~\ref{connect-glue-thm} and~\ref{module-thm}, we show that $\mathcal{R}_3$ is not included in a class of $c$-colorable graphs for a constant $c$. 

In Section~\ref{sec-ladder-graphs} we show that ladder graphs used in the study of prisms in~\cite{KP2008} are $2$-word-representable, and thus each ladder graph is a circle graph. Finally, in Section~\ref{crown-graphs-sec} we discuss $k$-word-representing comparability graphs via consideration of crown graphs, where we state some problems for further research.

\section{Preliminaries}

Suppose that $w$ is a word and $x$ and $y$  are two distinct letters in $w$. We say that $x$ and $y$ {\em alternate} in $w$ if after deleting in $w$ all letters but the copies of $x$ and $y$ we either obtain a word  $xyxy\cdots$ (of even or odd length) or a word  $yxyx\cdots$ (of even or odd length). If $x$ and $y$ do not alternate in $w$, we say that these letters are {\em non-alternating} in $w$. For example, if $w=31341232$ then the letters $1$ and $3$ are alternating in $w$ because removing all other letters we obtain $31313$, while $3$ and $4$ are non-alternating because removing all other letters from $w$  we have $3343$.

A simple graph $G=(V,E)$ is \emph{word-representable} if there exists a word
$w$ over the alphabet $V$ such that letters $x$ and $y$ alternate in
$w$ if and only if $(x,y)\in E$ for each $x\neq y$. We say that $w$ {\em represents} $G$, and $w$ is called a {\em word-representant} for $G$. The graphs in Figure \ref{fig:representable-graphs} are word-representable.  Indeed, for example,  1213423 is a word-representant for $M$, 1234 is a word-representant for $K_4$, and a word-representant for the Petersen graph is 
$$1387296(10)7493541283(10)7685(10)194562.$$ 

\begin{figure}[ht]
\begin{center}
\includegraphics[scale=0.6]{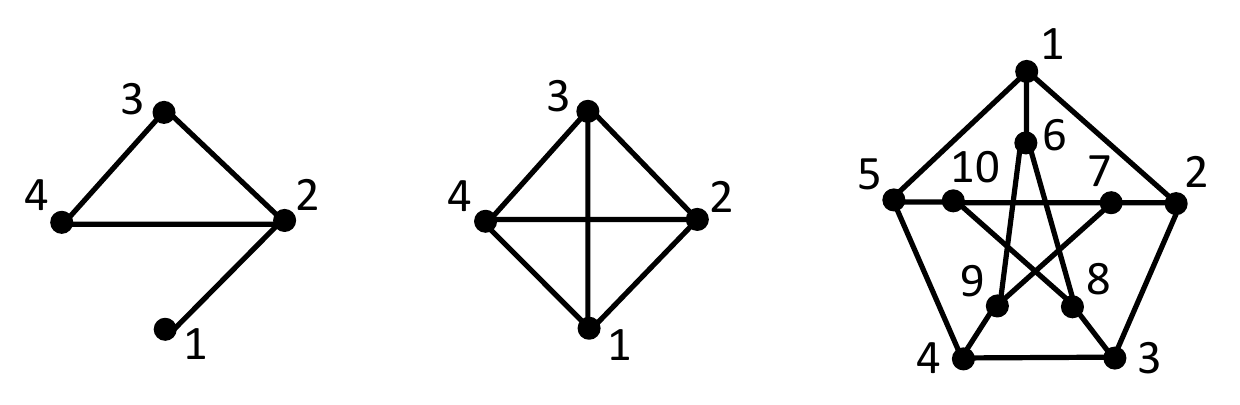}
\end{center}
\vspace{-20pt}
\caption{Three word-representable graphs $M$ (left), the complete graph $K_4$ (middle), and the Petersen graph (right).}
\label{fig:representable-graphs}
\end{figure}

In what follows, we will need the following two propositions that are easy to see from definitions.

\begin{proposition}\label{just-useful-observation}(\cite{KP2008})
Let $w=w_1xw_2xw_3$ be a word representing a graph $G$, where $w_1$, $w_2$ and $w_3$ are possibly empty words, and $w_2$ contains no $x$. Let $X$ be the set of all letters that appear only once in
$w_2$. Then possible candidates for $x$ to be adjacent to in $G$
are the letters in $X$.
\end{proposition}

\begin{proposition}\label{reverse-representant} If $w$ represents $G$, then the reverse $r(w)$, which is writing the letters in $w$ in the reverse order,  also represents $G$.\end{proposition}

\subsection{$k$-word-representable graphs}\label{k-word-representability}

A word is {\em $k$-uniform} if each letter in it appears exactly $k$ times. In particular, $1$-uniform words are permutations. A graph is \emph{$k$-word-representable} if there exists a $k$-uniform word representing it.  We say that the word {\em $k$-represents} the graph.

The following theorem is of special importance in the theory of word-representable graphs. 

\begin{theorem}\label{repr-must-be-k-repr}(\cite{KP2008}) A graph $G$  is word-representable if and only if it is $k$-word-representable
  for some $k$.  \end{theorem}

In what follows, we will also use the following two propositions. 

\begin{proposition}\label{k-implies-k-plus-1}(\cite{KP2008})   A $k$-word-representable
  graph $G$ is also $(k+1)$-word-representable. In particular, each word-representable graph has infinitely many word-representants representing it since for every $\ell>k$, a $k$-word-representable graph is also $\ell$-word-representable.\end{proposition}

\begin{proposition}(\cite{KP2008})\label{cyclic-shift} Let $w=uv$ be a $k$-uniform word representing a graph $G$, where $u$ and $v$ are two, possibly empty, words. Then the word $w'=vu$ also represents $G$.
\end{proposition}

\subsection{Permutationally representable graphs}

A graph $G$ with the vertex set $V=\{1,\ldots,n\}$ is {\em permutationally representable} if it can be
represented by a word of the form $p_1\cdots p_k$, where $p_i$
is a permutation of $V$ for $1\leq i\leq k$. For example, a complete graph $K_n$ is permutationally representable for any $n$. Indeed, take any permutation of $\{1,\ldots,n\}$ and repeat it as many times (maybe none) as desired. For another example, the path $1-2-3$ is also permutationally representable, and one such representation is 213231, while some other such representations can be obtained from it by adjoining any number of permutations 213 and 231 to it. If $G$ can be represented permutationally involving $k$ permutations, we say that $G$ is {\em permutationally $k$-representable}.

An orientation of graph's edges is called {\em transitive} if having (directed) edges $x\rightarrow y$ and $y\rightarrow z$ for some vertices $x,y,z$ implies having the edge $x\rightarrow z$. A graph is a {\em comparability graph} if it accepts a transitive orientation. 

\begin{theorem}(\cite{KS2008})\label{comparability-graphs} A graph is permutationally representable if and only if it is a comparability graph.  \end{theorem}

Comparability graphs correspond to partially ordered sets (posets), and the question whether such a graph can be permutationally $k$-represented is equivalent to the question whether the respective poset can be represented as an intersection of $k$ linear orders. The minimum such $k$ for a poset is called the {\em poset dimension}. We call {\em $k$-comparability graphs} those graphs that correspond to posets having dimension $k$, that is, those graphs that are permutationally $k$-representable but not permutationally $(k-1)$-representable.

\subsection{Semi-transitive orientations}\label{semi-transitive-orient-subsec}

A graph $G=(V,E)$ is \emph{semi-transitive} if it admits
an acyclic orientation such that for any directed path $v_1\rightarrow v_2\rightarrow \cdots \rightarrow v_k$ with $v_i\in V$ for all $i$, $1\leq i\leq k$, either
\begin{itemize}
\item there is no edge $v_1\rightarrow v_k$, or 
\item the edge $v_1\rightarrow v_k$ is present and there are edges $v_i\rightarrow v_j$ for all $1\le i<j\le k$. That is,  in this case, the (acyclic) subgraph induced by the vertices $v_1,\ldots,v_k$ is transitive (with the unique source $v_1$ and the unique sink $v_k$).  
\end{itemize}
We call such an orientation {\em semi-transitive orientation}. For example, the orientation of the graph in Figure \ref{fig:semi-transitive-orientation} is semi-transitive, and thus the underlying (non-directed) graph is semi-transitive. 

\begin{figure}[ht]
\begin{center}
\includegraphics[scale=0.6]{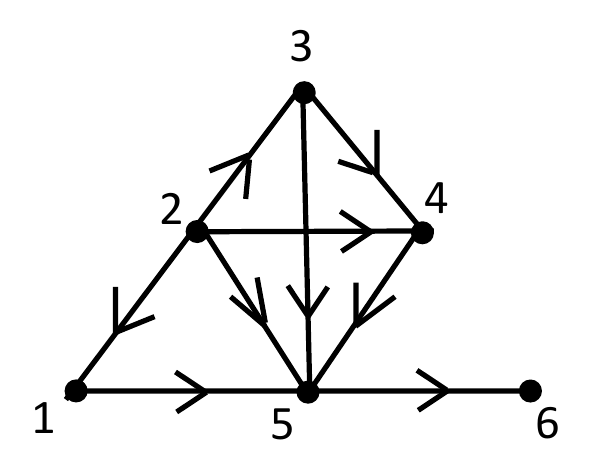}
\end{center}
\vspace{-20pt}
\caption{An example of a semi-transitive orientation.}
\label{fig:semi-transitive-orientation}
\end{figure}

Clearly,
all transitive (that is, comparability) graphs are semi-transitive,
and thus semi-transitive orientations are a generalization of transitive orientations. 

\begin{figure}[ht]
\begin{center}
\includegraphics[scale=0.6]{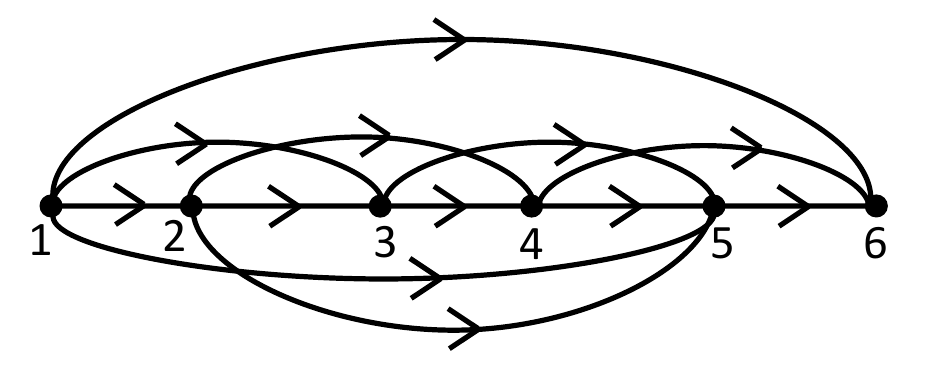}
\end{center}
\vspace{-20pt}
\caption{An example of a shortcut.}
\label{fig:shortcut}
\end{figure}

Semi-transitive orientations are defined in \cite{HKP2011} in terms of {\em shortcuts} as follows. A \emph{semi-cycle} is the directed acyclic
graph obtained by reversing the direction of one edge of a directed
cycle. An acyclic digraph is a shortcut if it is induced by
the vertices of a semi-cycle and contains a pair of non-adjacent
vertices. Thus, a digraph on the vertex set $\{ v_1, \ldots,
v_k\}$ is a shortcut if it contains a directed path $v_1\rightarrow v_2\rightarrow \cdots
\rightarrow v_k$, the edge $v_1\rightarrow v_k$, and it is missing an edge $v_i\rightarrow v_j$ for some $1 \le i
< j \le k$; in particular, we must have $k\geq 4$, so that any shortcut is on at least four vertices. See Figure~\ref{fig:shortcut} for an example of a shortcut (there, the edges $1\rightarrow 4$, $2\rightarrow 6$, and $3\rightarrow 6$ are missing). An orientation of a graph is semi-transitive, if it is acyclic and contains no
shortcuts. Clearly, this definition is just another way to introduce the notion of semi-transitive orientations presented above. 

The following theorem is a useful characterization of word-representable graphs that allows answering questions on word-representability in terms of graph orientations.  

\begin{theorem}\label{thm:rep-equals-semi-trans}(\cite{HKP2011})
A graph $G$ is word-representable if and only if it is semi-transitive (that is, it accepts a semi-transitive orientation).
\end{theorem}

A direct corollary to the last theorem is the following statement.

\begin{theorem}(\cite{HKP2011})\label{obs:3col} $3$-colorable graphs are word-representable.\end{theorem}

\subsection{Graph's representation number}\label{graphs-repr-number}

The following statement is easy to see. 

\begin{theorem}\label{1-repr-characterization} A graph $G$ is in  $\mathcal{R}_1$ if and only if $G=K_n$, the complete graph on $n$ vertices, for some $n$.\end{theorem}

We also have a characterization of graphs in $\mathcal{R}_2$. To state it we need the following definition:  A {\em circle graph} is an undirected graph whose vertices can be associated with chords of a circle such that two vertices are adjacent if and only if the corresponding chords cross each other.

\begin{theorem}\label{2-repr-characterization}(\cite{HKP2011}) For a graph $G$ different from a complete graph, $G$ is in $\mathcal{R}_2$ if and only if $G$ is a circle graph. \end{theorem}

The following proposition is not difficult to see from  definitions.

\begin{proposition}\label{useful-remark-ind-subgraph} If $\mathcal{R}(G)=k$ and $G'$ is an induced subgraph of $G$ then $\mathcal{R}(G')\leq k$. \end{proposition}

\begin{proof} Indeed, using the hereditary nature of word-representable graphs, if representing $G'$ would require more than $k$ copies of each letter, then representing $G$ would obviously require more than $k$ copies of each letter.\end{proof}

\begin{theorem}(\cite{HKP2011})\label{thm:di-to-string}
Each word-representable graph on $n$ vertices is $n$-word-representable.
\end{theorem}

\begin{figure}[ht]
\begin{center}
\includegraphics[scale=0.6]{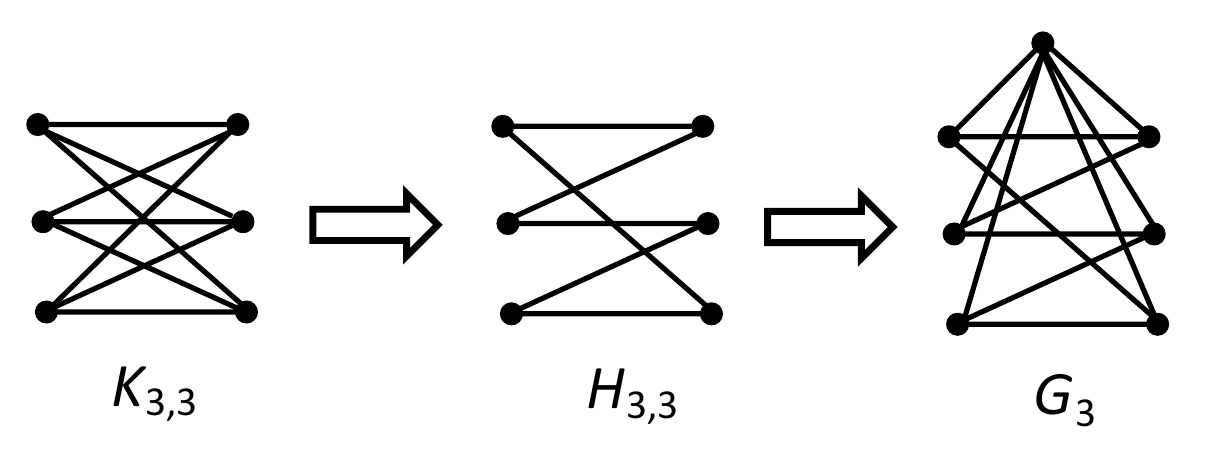}
\end{center}
\vspace{-20pt}
\caption{Graphs involved in defining the graph $G_3$.}
\label{fig:graph-G-k}
\end{figure}

It turns out that there are graphs on $n$ vertices with representation number of $\lfloor\frac{n}{2}\rfloor$,
matching the upper bound in Theorem \ref{thm:di-to-string} within a factor of 2. A {\em crown graph} $H_{k,k}$ is a graph obtained from
the complete bipartite graph $K_{k,k}$ by removing a perfect
matching (see Figure~\ref{crown-pic} for examples of crown graphs). Then $G_k$ is the graph obtained from a crown
graph $H_{k,k}$ by adding an all-adjacent vertex.  See Figure \ref{fig:graph-G-k} for graphs involved in defining $G_3$, where note that all choices in defining $H_{3,3}$ lead to isomorphic graphs.

\begin{theorem}(\cite{HKP2011})\label{thm:example}
The graph $G_k$ on $2k+1$ vertices belongs to $\mathcal{R}_k$.
\end{theorem}

\subsection{Two operations to extend a graph}

The following theorem gives a useful tool to construct $3$-word-representable graphs, that is, graphs with representation number at most $3$.

\begin{theorem}\label{add-path-to-graph}(\cite{KP2008})
Let $G=(V,E)$ be a $3$-word-representable graph and $x,y\in V$. Denote by
$H$ the graph obtained from $G$ by adding to it a path of length
at least $3$ connecting $x$ and $y$. Then $H$ is also
$3$-word-representable.
\end{theorem}

We also have the following proposition.

\begin{proposition}\label{adding-leaf} Let $G\in\mathcal{R}_k$, where $k\geq 2$, and $x\in V(G)$. Also, let $G'$ be the graph obtained from $G$ by adding an edge $(x,y)$, where $y\not\in V(G)$. Then $G'\in\mathcal{R}_k$. \end{proposition}

\begin{proof} Suppose that $G$ is $k$-represented by a word $w_0xw_1xw_2\cdots xw_{k-1}xw_k$, where for $0\leq i\leq k$, $w_i$ is a word not containing $x$. Then it is not difficult to check that the word $$w_0yxyw_1xw_2yxw_3yxw_4\cdots yxw_{k-1}yxw_k$$
$k$-represents $G'$ (in particular, the vertex $x$ is the only neighbor of $y$). Finally, if $G'$ could be $(k-1)$-represented by some word, we would remove from that word the letter $y$ to obtain a $(k-1)$-representation of $G$, which is impossible. So, $G'\in\mathcal{R}_k$.\end{proof}

\section{Prisms and $\mathcal{R}_3$}

A {\em prism} $\mbox{Pr}_n$ is a graph consisting of two cycles $12\cdots n$ and $1'2'\cdots n'$, where $n\geq 3$, connected by the edges $(i,i')$ for $i=1,2,\ldots ,n$. In particular, the $3$-dimensional cube is a prism. Examples of prisms are given in Figure~\ref{prisms-pic}. The leftmost prism there is called the {\em triangular prism}. The middle prism is the $3$-dimensional cube.

\begin{figure}[ht]
\begin{center}
\includegraphics[scale=0.6]{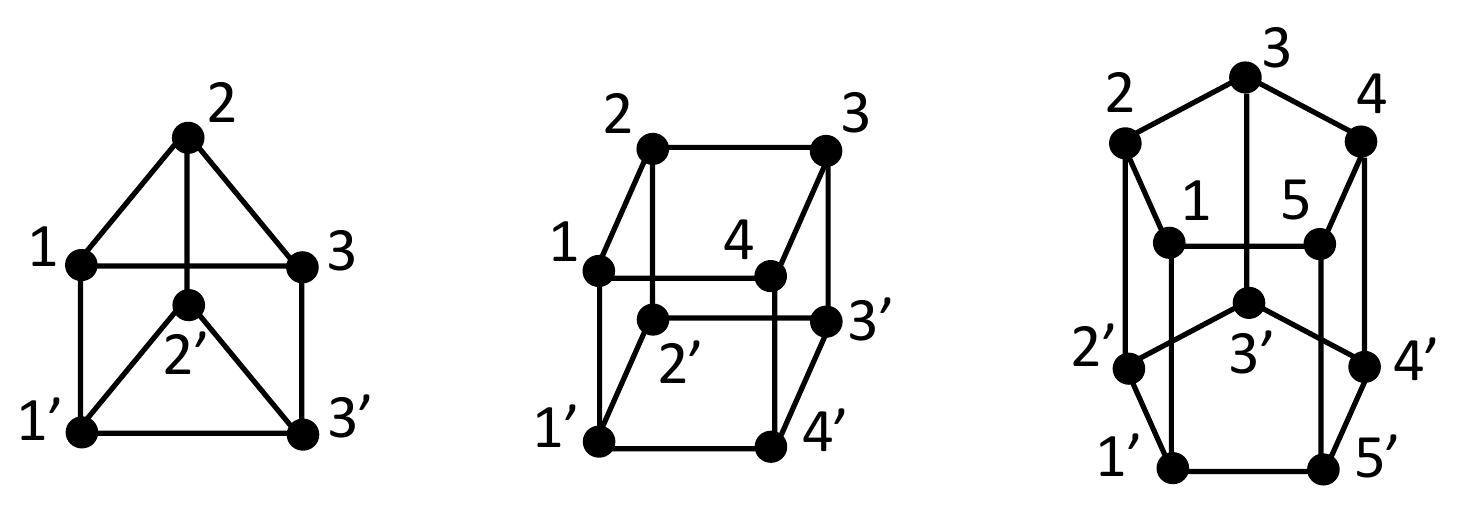}
\end{center}
\vspace{-20pt}
\caption{Examples of prisms.}
\label{prisms-pic}
\end{figure}

\begin{theorem}\label{every-prism-3-repr}(\cite{KP2008}) Every prism $\mbox{Pr}_n$ is $3$-word-representable. \end{theorem}

\begin{theorem}\label{Pr3-thm}(\cite{KP2008}) The triangular prism $\mbox{Pr}_3$ is not $2$-word-representable, and thus, by Theorem~\ref{every-prism-3-repr}, $\mathcal{R}(\mbox{Pr}_3)=3$.\end{theorem}

Our next theorem extends Theorem~\ref{Pr3-thm} by showing that {\em any} prism belongs to $\mathcal{R}_3$.

\begin{theorem}\label{Prn-repr-num-3-thm} For $n\geq 4$, $\mbox{Pr}_n$ is not $2$-word-representable, and thus, by Theorem~\ref{every-prism-3-repr}, $\mathcal{R}(\mbox{Pr}_n)=~3$. \end{theorem}

\begin{proof} Suppose that $\mbox{Pr}_n$  can be $2$-represented by a word $w$ for $n\geq 4$. It is not difficult to see that there must exist a letter $x$ in $w$ such that no letter occurs twice between the two copies of $x$. By Proposition~\ref{just-useful-observation}, there are exactly three letters between the copies of $x$. Using symmetry,  Proposition~\ref{cyclic-shift} and Proposition~\ref{reverse-representant}, we only need to consider two cases (the second one is unnecessary in the case of $n=4$ because of symmetry) where we took into account that the vertices $1'$, $2$ and $n$ form an independent set:  
\begin{itemize}
\item $w$ is of the form $11'2n1\cdots n\cdots 2\cdots 1'\cdots$.  Since $(n,n')\in E(\mbox{Pr}_n)$ and $(1',n')\in E(\mbox{Pr}_n)$, we can refine the structure of $w$ as follows
$$w=11'2n1\cdots n'\cdots n\cdots 2\cdots 1'\cdots n'\cdots.$$ However, $2$ and $n'$ alternate in $w$ contradicting to the fact that $(2,n')\not\in E(\mbox{Pr}_n)$.
\item $w$ is of the form $121'n1\cdots n\cdots 1'\cdots 2\cdots$. In this case, we will refine the structure of $w$ in two different ways and then will merge these refinements:
\begin{itemize}
\item Since $(2,2')\in E(\mbox{Pr}_n)$, $(1',2')\in E(\mbox{Pr}_n)$ and $(2',n)\not\in E(\mbox{Pr}_n)$, $w$ must be of the form
$$w=  121'n1\cdots n\cdots 2'\cdots 1'\cdots 2\cdots 2'\cdots.$$
\item Since $(n,n')\in E(\mbox{Pr}_n)$, $(1',n')\in E(\mbox{Pr}_n)$ and $(2,n')\not\in E(\mbox{Pr}_n)$, $w$ must be of the form
$$w=121'n1\cdots n'\cdots n\cdots 1'\cdots n'\cdots 2\cdots$$
\end{itemize}

Merging the refinements, we see that $w$ must be of the form 
$$w=121'n1\cdots n'\cdots n\cdots 2' \cdots 1'\cdots n'\cdots 2\cdots 2'\cdots.$$
However, we see that the letters $2'$ and $n'$ alternate in $w$ contradicting to the fact that $(2',n')\not\in E(\mbox{Pr}_n)$.
\end{itemize}
\vspace{-7mm}\end{proof}

Applying  Proposition~\ref{adding-leaf} as many times as necessary, we see that if $G\in \mathcal{R}_3$ then a graph obtained from $G$ by attaching simple paths of any lengths to vertices in $G$ belongs to $\mathcal{R}_3$; call this way to extend graphs ``operation 1''. Also, by Theorem~\ref{add-path-to-graph}, we can add simple paths of length at least $3$ connecting any pair of vertices in $G\in \mathcal{R}_3$ and still obtain a graph in $\mathcal{R}_3$ (if a $2$-word-representable graph would be obtained, we would have a contradiction with $G\in \mathcal{R}_3$); call this way to extend graphs ``operation 2''. The only known to us for the moment graphs in $\mathcal{R}_3$ are recorded in the following statement whose truth follows from considerations above.

\begin{theorem}\label{known-R3}  $\mathcal{R}_3$  contains the Petersen graph, prisms, and any other graph obtained from these by applying operations $1$ and $2$ arbitrary number of times in any order. \end{theorem}

\section{Connecting two graphs by an edge and gluing two graphs in a vertex}\label{glueing-connecting-two-graphs}

The operations of connecting two graphs, $G_1$ and $G_2$, by an edge and gluing these graphs in a vertex are presented schematically in Figure~\ref{glueing-connecting}. It follows directly from Theorem~\ref{thm:rep-equals-semi-trans} that if both $G_1$ and $G_2$ are word-representable then the resulting graphs will be word-representable too, while if at least one of $G_1$ or $G_2$ is non-word-representable then the resulting graphs will be non-word-representable. Indeed, if $G_1$ and $G_2$ are oriented semi-transitively, then orienting the edge $(x,y)$ in either direction will not give a chance for the resulting graph to have a shortcut (defined in Subsection~\ref{semi-transitive-orient-subsec}) thus resulting in a semi-transitively oriented graph; similarly, no shortcut is possible when semi-transitively oriented $G_1$ and $G_2$ are glued in a vertex $z$. 

\begin{figure}[ht]
\begin{center}
\includegraphics[scale=0.6]{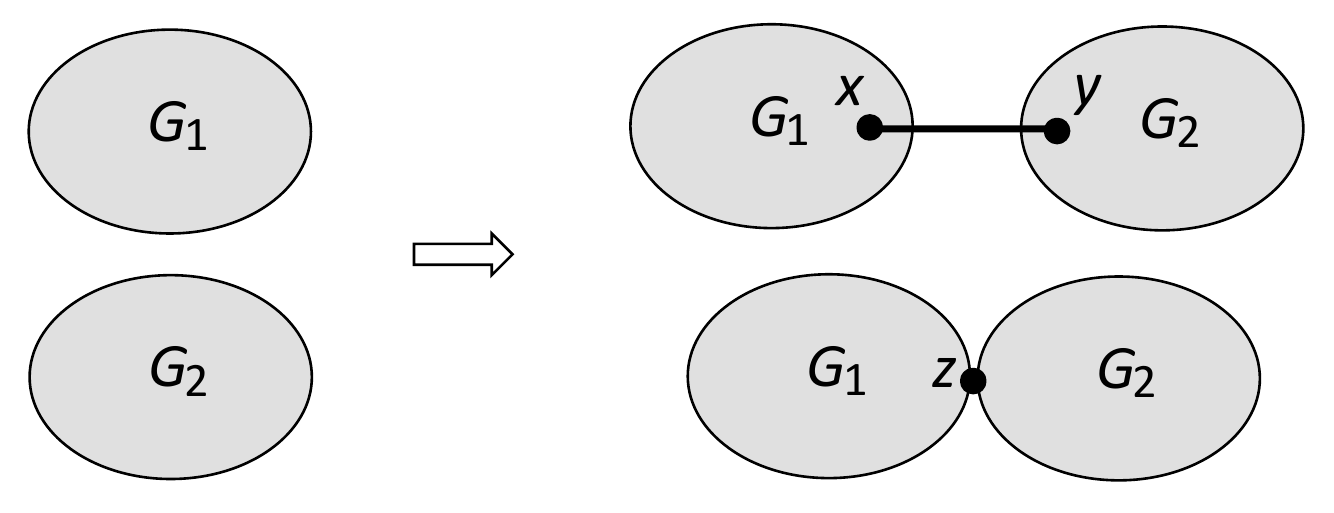}
\end{center}
\vspace{-20pt}
\caption{Connecting graphs by an edge and gluing graphs in a vertex.}
\label{glueing-connecting}
\end{figure}

While the arguments above involving orientations answer the question on word-representability of connecting graphs by an edge or gluing graphs in a vertex, they do not allow to answer the following question: If $G_1$ is $k_1$-word-representable, $G_2$ is $k_2$-word-representable, and $G$ is $k$-word-representable (such a $k$ must exist by Theorem~\ref{repr-must-be-k-repr}) then what can be said about $k$? Theorem~\ref{connect-glue-thm} below, that is based on Theorems~\ref{connect-by-edge} and~\ref{glue-in-vertex}, answers this question.

\begin{theorem}\label{connect-by-edge}(\cite{KP2008})
For $k\geq 2$, let $w_1$ and $w_2$ be $k$-uniform words representing
graphs $G_1=(V_1,E_1)$ and $G_2=(V_2,E_2)$, respectively, where $V_1$
and $V_2$ are disjoint. Suppose that $x\in V_1$ and $y\in V_2$. Let
$H_1$ be the graph $(V_1\cup V_2,E_1\cup E_2\cup\{(x,y)\})$. Then $H_1$ is $k$-word-representable.
\end{theorem}

\begin{theorem}\label{glue-in-vertex}(\cite{KP2008})
For $k\geq 2$, let $w_1$ and $w_2$ be $k$-uniform words representing
graphs $G_1=(V_1,E_1)$ and $G_2=(V_2,E_2)$, respectively, where $V_1$
and $V_2$ are disjoint. Suppose that $x\in V_1$ and $y\in V_2$. Let $H_2$ be the graph obtained from $G_1$ and $G_2$ by
identifying $x$ and $y$ into a new vertex~$z$. Then $H_2$ is $k$-word-representable.
\end{theorem}

\begin{theorem}\label{connect-glue-thm} Suppose that for graphs $G_1=(V_1,E_1)$ and $G_2=(V_2,E_2)$, $\mathcal{R}(G_1)=k_1$ and $\mathcal{R}(G_2)=k_2$, $x\in V_1$, $y\in V_2$ and $k=\max(k_1,k_2)$.  Also, let the graph $G'$ be obtained by connecting $G_1$ and $G_2$ by  the edge $(x,y)$, and the graph $G''$ be obtained from $G_1$ and $G_2$ by identifying the vertices $x$ and $y$ into a single vertex $z$. The following holds. 

\begin{enumerate}
\item If $|V_1|=|V_2|=1$ then both $G'$ and $G''$ are cliques and thus $k_1=k_2=1$. In this case, $\mathcal{R}(G')=\mathcal{R}(G'')=1$.
\item If $\min(|V_1|,|V_2|)=1$ but $\max(|V_1|,|V_2|)>1$ then $\mathcal{R}(G'')=k$ and $\mathcal{R}(G')=\max(k,2)$.   
\item If $\min(|V_1|,|V_2|)>1$  then $\mathcal{R}(G')=\mathcal{R}(G'')=\max(k,2)$. 
\end{enumerate}
\end{theorem}

\begin{proof} The first part of the statement is easy to see since both a single vertex ($G''$) and the one edge graph ($G'$) are $1$-word-representable by Theorem~\ref{1-repr-characterization}. 

For part 2, without loss of generality, $|V_1|=1$ (that is, $V_1=\{x\}$) and thus $G''=G_2$ leading to $\mathcal{R}(G'')=k$. On the other hand, $G'$ is not a clique and thus $\mathcal{R}(G')\geq 2$. If $k_2=1$ then $G_2$ can be represented by the permutation $yy_1\cdots y_{|V_2|-1}$ for $y, y_i\in V_2$ and thus $G'$ can be represented by $xyxy_1\cdots y_{|V_2|-1}yy_1\cdots y_{|V_2|-1}$ leading to $\mathcal{R}(G')=2$. However, if $k_2\geq 2$, so that $k=k_2$, we can take any $k$-word-representation of $G_2$ and replace in it every other occurrence of the letter $y$ by $xyx$ to obtain a $k$-word-representation of $G'$. Thus,  $\mathcal{R}(G')=k$ because if it would be less than $k$, we would have $\mathcal{R}(G_2)<k$ by Proposition~\ref{useful-remark-ind-subgraph}, a contradiction. 

For part 3, neither $G'$ nor $G''$ is a clique and thus $\mathcal{R}(G'),\mathcal{R}(G'')\geq 2$. By Proposition~\ref{k-implies-k-plus-1}, both $G_1$ and $G_2$ are $k$-word-representable. If $k\geq 2$ then by Theorems~\ref{connect-by-edge} and~\ref{glue-in-vertex} both $G'$ and $G''$ are $k$-word-representable leading to $\mathcal{R}(G')=\mathcal{R}(G'')=k$ since if $\mathcal{R}(G')<k$ or $\mathcal{R}(G'')<k$ we would obtain a contradiction either with $\mathcal{R}(G_1)=k_1$ or with $\mathcal{R}(G_2)=k_2$ by Proposition~\ref{useful-remark-ind-subgraph}. Finally, if $k=1$ then $G_1$ and $G_2$ must be cliques that can be represented by permutations $x_1\cdots x_{|V_1|-1}x$ and $yy_1\cdots y_{|V_2|-1}$, respectively, for $x, x_i\in V_1$ and $y, y_i\in V_2$. Then the words $$x_1\cdots x_{|V_1|-1}xx_1\cdots x_{|V_1|-1}yxy_1\cdots y_{|V_2|-1}yy_1\cdots y_{|V_2|-1}$$ and
$$x_1\cdots x_{|V_1|-1}zx_1\cdots x_{|V_1|-1}y_1\cdots y_{|V_2|-1}zy_1\cdots y_{|V_2|-1}$$ 2-word-represent the graphs $G'$ and $G''$, respectively, and thus  $\mathcal{R}(G')=\mathcal{R}(G'')=2$.
 \end{proof}

\section{Replacing a vertex in a graph with a module}\label{modules-subsec}

A subset $X$ of the set of vetrices $V$ of a graph $G$ is a {\em module} if all members of $X$ have the same set of neighbors among vertices not in $X$ (that is, among vertices in $V\setminus X$). For example, Figure~\ref{modules-business} shows replacing the vertex $1$ in the triangular prism by the module $K_3$ formed by the vertices $a$, $b$ and $c$.

\begin{figure}[ht]
\begin{center}
\includegraphics[scale=0.6]{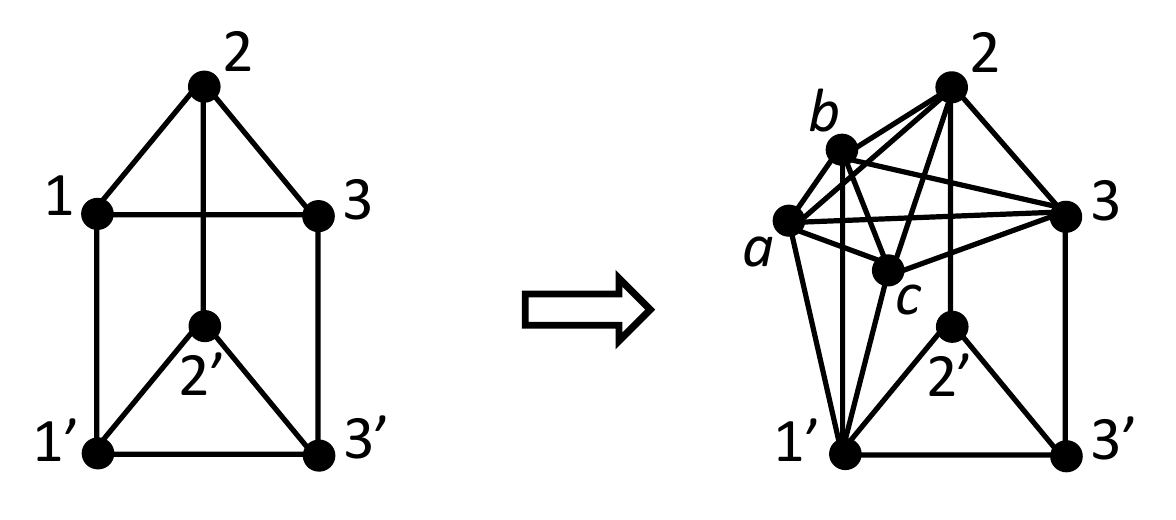}
\end{center}
\vspace{-20pt}
\caption{Replacing a vertex by a module.}
\label{modules-business}
\end{figure}

The following theorem is an extended version of an observation made in \cite{HKP2010}.

\begin{theorem}\label{module-thm} Suppose that $G$ is a word-representable graph and $x\in V(G)$. Let $G'$ be obtained from $G$ by replacing $x$ with a module $M$, where $M$ is any comparability graph (in particular, any clique). Then $G'$ is also word-representable. Moreover, if $\mathcal{R}(G)=k_1$ and $\mathcal{R}(M)=k_2$ then $\mathcal{R}(G')=k$, where $k=\max\{k_1,k_2\}$. \end{theorem}

\begin{proof}  By Theorem~\ref{comparability-graphs}, $M$ can be represented by a word $p_1\cdots p_{k_2}$, where $p_i$ is a permutation (of length equal to the number of vertices in $M$) for $1\leq i\leq k_2$. If $k_2<k$, we can adjoin to the representation of $M$ any number of copies of permutations in the set $\Pi=\{p_1,\ldots,p_{k_2}\}$ to obtain a $k$-representation $p=p_1\cdots p_k$ of $M$, where $p_i\in \Pi$ (indeed, no alternation properties will be changed while adjoining such extra permutations). 

Using Proposition~\ref{k-implies-k-plus-1}, if necessary (in case when $k_1<k$), we can assume that $G$ can be $k$-represented by a word $w=w_0xw_1xw_2\cdots xw_k$ for some words $w_i$ not containing $x$ for all $0\leq i\leq k$. But then the word $w'=w_0p_1w_1p_2w_2\cdots p_kw_k$ $k$-represents $G'$. Indeed, it is easy to see that the pairs of letters in $w'$ from $V(M)$ have the right alternation properties, as do the pairs of letters from $V(G)\setminus \{x\}$ in $w'$.  On the other hand, it is not difficult to see that if $y\in V(M)$ and $z\in V(G)\setminus \{x\}$ then $(y,z)\in E(G')$ if and only if $(x,z)\in E(G)$.    

If $G'$ would be $(k-1)$-word-representable, we would either obtain a contradiction with $\mathcal{R}(G)=k_1$ (after replacing each $p_i$ in $w'$ with $x$) or with $\mathcal{R}(M)=k_2$ (after removing all letters in $w'$ that are not in $M$). Thus,  $\mathcal{R}(G')=k$.\end{proof}

\section{Graphs in $\mathcal{R}_3$ and $c$-colorable graphs}\label{sec-c-colorable-graphs}

Theorems~\ref{1-repr-characterization} and~\ref{2-repr-characterization} show that there exists no constant $c$ such that all $1$- or $2$-word-representable graphs are $c$-colorable. Indeed, $K_n$ is $n$-colorable, while circle graphs formed by $K_{n-1}$ and an isolated vertex  (which are not $1$-word-representable) are $(n-1)$-colorable. On the other hand, known to us graphs $G_k$ (considered in Theorem~\ref{thm:example}) that require many copies of each letter to be represented, are $3$-colorable for any $k\geq 2$ (because $G_k$ is a bipartite graph with an all-adjacent vertex).  Thus, $3$-word-representable graphs do not contain a class of $c$-colorable graphs for some constant $c\geq 3$ (this claim follows from the fact that $G_k$, being $3$-colorable, is not 3-word-representable for $k\geq 4$). 

A natural question to ask here is: Is $\mathcal{R}_3$ properly included in a class of $c$-colorable graphs for a constant $c$? The following theorem shows that this is not the case.

\begin{theorem}\label{thm-R3-c-colorable} The class  $\mathcal{R}_3$ is not included in a class of $c$-colorable graphs for some constant~$c$. \end{theorem}

\begin{proof} Suppose that any graph in class $\mathcal{R}_3$ is $c$-colorable for some constant $c$. We can assume that $c\geq 3$ since $G_3$ being $3$-colorable belongs to $\mathcal{R}_3$ by Theorem~\ref{thm:example}. Consider the triangular prism $Pr_3$ to the left in Figure~\ref{modules-business}, which is also $3$-colorable and, by Theorem~\ref{Pr3-thm}, belongs to $\mathcal{R}_3$. Replace the vertex $1$ in $Pr_3$ with a module $K_{c+1}$, the complete graph on $c+1$ vertices (see Section~\ref{modules-subsec} for the notion of a module) as shown for the case of $c=2$ in Figure~\ref{modules-business}. Denote the obtained graph by $Pr'_3$. Since $\mathcal{R}(K_{c+1})=1$, by Theorem~\ref{module-thm}, $Pr'_3\in\mathcal{R}_3$. However, $Pr'_3$ is not $c$-colorable since it contains a clique of size $c+1$ ($Pr'_3$ is $(c+1)$-colorable). We obtain a contradiction with our assumption.\end{proof}

We note that an alternative way to obtain a contradiction in the proof of Theorem~\ref{thm-R3-c-colorable} is to consider the triangular prism $Pr_3$ (which is in $\mathcal{R}_3$),  the complete graph $K_{c+1}$ (which is in $\mathcal{R}_1$), and either to connect these graphs by an edge, or glue these graphs in a vertex. Then by  Theorem~\ref{connect-glue-thm}, the obtained graph will be in $\mathcal{R}_3$, but it is $(c+1)$-colorable.

From considerations in this section it follows that, in particular, the classes $\mathcal{R}_3$ and $3$-colorable graphs (that are word-representable by Theorem~\ref{obs:3col}) are not comparable in the sense that none of these classes is included in the other one. However, this section does not answer the following question.

\begin{problem}\label{problem0} Can each bipartite (that is, $2$-colorable) graph be $3$-word-represented? Namely, does each bipartite graph belong to the union of the sets $\mathcal{R}_1 \cup\mathcal{R}_2 \cup \mathcal{R}_3$? \end{problem}

We suspect the answer to the question in Problem~\ref{problem0} to be negative. A good candidate for a counterexample should be the crown graph $H_{k,k}$ for $k\geq 5$; see Section~\ref{crown-graphs-sec} and, in particular, Problem~\ref{problem2}.

\section{Ladder graphs}\label{sec-ladder-graphs}

The {\em ladder graph} $L_n$ with $2n$ vertices and $3n-2$ edges is presented in Figure~\ref{ladder-pic}.

\begin{figure}[ht]
\begin{center}
\includegraphics[scale=0.6]{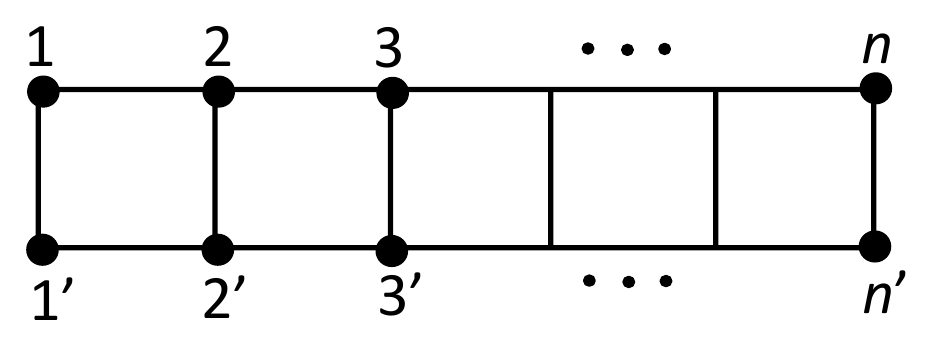}
\end{center}
\vspace{-20pt}
\caption{The ladder graph $L_n$.}
\label{ladder-pic}
\end{figure}

It follows from the proof of Theorem~\ref{every-prism-3-repr} in \cite{KP2008} that $L_n$ is 3-word-representable. In the following theorem, we will show that $L_n$ is actually 2-word-representable for $n\geq 2$ (it is clearly $1$-word-representable for $n=1$).

\begin{theorem}\label{ladder-graphs-are-2-repr}  For $n\geq 2$, $\mathcal{R}(L_n)=2$. Thus, by Theorem~\ref{2-repr-characterization}, any ladder graph is a circle graph. \end{theorem}

\begin{proof}  We prove the statement by induction on $n$. $L_1$ can be 2-represented by the word $w_1=11'11'$ that has the factor $1'1$. Substituting the factor $1'1$ in $w_1$ by $2'1'22'12$ and reversing the entire word, we obtain the word $w_2=1'212'21'2'1$ that contains the factor $2'2$. It is  straightforward to check that $w_2$ represents $L_2$ since one only needs to check the alternation properties of the just added letters $2$ and~$2'$.  

More  generally, given a $2$-representation $w_i$ of $L_i$ containing the factor $i'i$, we substitute $i'i$ in $w_i$ by $(i+1)'i'(i+1)(i+1)'i(i+1)$ and reverse the entire word to obtain the word $w_{i+1}$ containing the factor $(i+1)'(i+1)$. It is straightforward to check that $w_{i+1}$ represents $L_{i+1}$ since the only thing that needs to be checked is the right alternation properties of the just added letters $i+1$ and $(i+1)'$. We are done. \end{proof}

 In Table~\ref{ex-2-repr-Ln}, we record $2$-representations of the ladder graph $L_n$ for $n=1,\ldots,5$. The factors $n'n$ are indicated in bold.

\begin{table}
\begin{center}
\begin{tabular}{c|c}
$n$ & $2$-representation of the ladder graph $L_n$ \\
\hline
1 &  $1{\bf 1'1}1'$\\
\hline
2 &  $1'21 {\bf 2'2} 1'2'1$\\
\hline
3 & $12'1'32 {\bf 3'3} 2'3'121'$\\ 
\hline
4 & $1'213'2'43 {\bf 4'4} 3'4'231'2'1$\\
\hline
5 & $12'1'324'3'54{\bf 5'5}4'5'342'3'121'$\\ 
\hline
\end{tabular}
\caption{$2$-representation of the ladder graph $L_n$ for $n=1,\ldots,5$.}\label{ex-2-repr-Ln}
\end{center}
\end{table}

Of course, an alternative proof of Theorem~\ref{ladder-graphs-are-2-repr} would be in representing vertices in $L_n$ by properly overlapping chords on a circle. However, we find our proof of that theorem a bit easier to follow/record.  

\begin{figure}[ht]
\begin{center}
\includegraphics[scale=0.6]{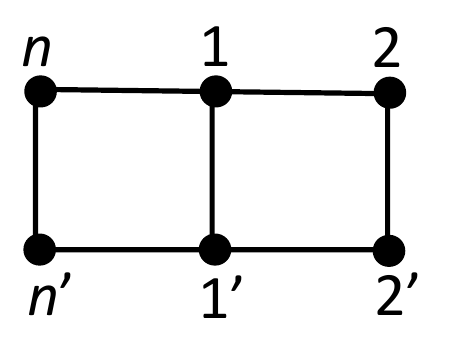}
\end{center}
\vspace{-20pt}
\caption{The ladder graph on $6$ vertices.}
\label{ladder-3-pic}
\end{figure}

\begin{remark} The thoughtful Reader would notice what seems to be an inconsistency: our argument in the proof of Theorem~\ref{Prn-repr-num-3-thm} seems to show that the graph in Figure~\ref{ladder-3-pic} is not $2$-word-representable, while by Theorem~\ref{ladder-graphs-are-2-repr} we see that this graph is $2$-word-representable by 
$n1'n'212'21'2'n1n'$ (we renamed the labels in $L_3$ respectively, and used the third line in Table~\ref{ex-2-repr-Ln}). The reason for possible confusion is that we used symmetry in the proof of  Theorem~\ref{Prn-repr-num-3-thm} to assume that there are exactly three letters between the $1$s; such an assumption cannot be made while dealing with the graph in Figure~\ref{ladder-3-pic} because $1$ is an ``internal'' vertex there, while there are ``external'' vertices as well, namely $2,2',n,n'$.
\end{remark}

\section{Crown graphs}\label{crown-graphs-sec}

Recall definition of the crown graph $H_{k,k}$ in Section~\ref{graphs-repr-number} and see Figure~\ref{crown-pic} for a few small such graphs. It is a well-known fact that the dimension of the poset corresponding to $H_{k,k}$ is $k$ for $k\geq 2$, and thus $H_{k,k}$  is a $k$-comparability graph (it is permutationally $k$-representable but not permutationally $(k-1)$-representable). In \cite{HKP2011}, the following way to represent permutationally $H_{k,k}$ was suggested. Concatenate the permutation $12\cdots (k-1)k'k(k-1)'\cdots 2'1'$ together with all permutations obtained from this by simultaneous exchange of $k$ and $k'$ with $m$ and $m'$, respectively, for $m=1,\ldots,k-1$. See Table~\ref{ex-k-repr-Hkk} for permutationally $k$-representation of $H_{k,k}$ for $k=1,2,3,4$.

\begin{figure}[ht]
\begin{center}
\includegraphics[scale=0.6]{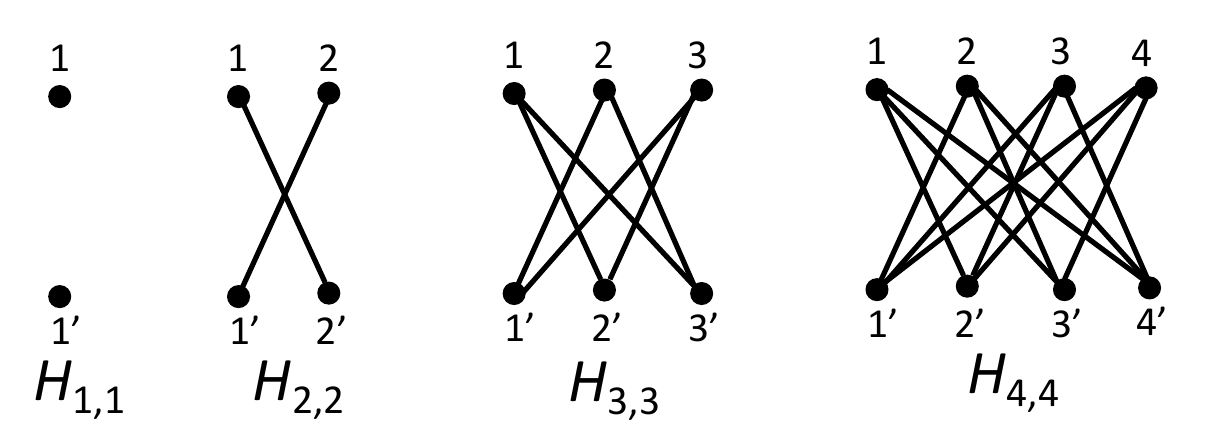}
\end{center}
\vspace{-20pt}
\caption{The crown graph $H_{k,k}$ for $k=1,2,3,4$.}
\label{crown-pic}
\end{figure}

It is not difficult to see by induction on the number of vertices, that each tree is 2-word-representable (just add a new leaf $y$ to a vertex $x$ in a tree $T$, and substitute an $x$ in the word representing $T$ with $yxy$). Based on this, one can show that any cycle graph is 2-word-represnetable (one first represents a path, which is a tree, and then adds one more edge by swapping a certain pair of consecutive letters in the word representing the path). 2-representing trees and cycles is discussed in detail in \cite{KL}.

\begin{table}
\begin{center}
\begin{tabular}{c|c}
$k$ & permutationally $k$-representation of the crown graph $H_{k,k}$\\
\hline
1 & $11'1'1$\\
\hline
2 & $12'21'21'12'$\\
\hline
3 & $123'32'1'132'23'1'231'13'2'$\\ 
\hline
4 & $1234'43'2'1'1243'34'2'1'1342'24'3'1'2341'14'3'2'$\\ 
\hline
\end{tabular}
\caption{Permutationally $k$-representation of the crown graph $H_{k,k}$ for $k=1,2,3,4$.}\label{ex-k-repr-Hkk}
\end{center}
\end{table}

Consider the cycle graph $C_5$ on five edges. $C_5$ is in $\mathcal{R}_2$ (it is 2-word-representable, but not 1-word-representable), but it is not a comparability graph and thus, by Theorem~\ref{comparability-graphs}, $C_5$ is not permutationally representable. On the other hand, the crown graphs $H_{1,1}$ and $H_{2,2}$, being $2$-comparability graphs, belong to $\mathcal{R}_2$. Also, $H_{3,3}$, being a $3$-comparability graph, belongs to $\mathcal{R}_2$, since $H_{3,3}$ is the cycle graph $C_6$. Moreover, $H_{4,4}$, being a 4-comparability graph, belongs to $\mathcal{R}_4$, which follows from the fact that  $H_{4,4}$ is the prism $Pr_4$ (the $3$-dimensional cube) and Theorem~\ref{Prn-repr-num-3-thm} can be applied.

Crown graphs, being bipartite graphs, and thus comparability graphs, provide an interesting case study of relations between $k$-comparability graphs and $k$-word-representable graphs. While each $k$-comparability graph is necessarily $k$-word-representable, in some cases such a graph is also $(k-1)$-word-representable, and, in fact, it is in $\mathcal{R}_{k-1}$ in the known to us situations.  Thus, it seems like giving up permutational representability, we should be able to come up with a shorter representation of a given comparability graph. However, we do not know whether this is essentially always the case (except for some particular cases like the graphs $H_{1,1}$ and $H_{2,2}$). Thus we state the following open problem. 

\begin{problem}\label{problem1} Characterize $k$-comparability graphs that belong to $\mathcal{R}_{k-\ell}$ for a fixed  $\ell$. In particular, characterize those $k$-comparability graphs that belong to $\mathcal{R}_{k-1}$. Is the set of $k$-comparability graphs that belong to  $\mathcal{R}_{k-\ell}$ (non-)empty for a fixed $\ell\geq 2$?\end{problem}

A step towards solving Problem~\ref{problem1} could be first understanding crown graphs and solving the following problem. 

\begin{problem}\label{problem2} Characterize those $H_{k,k}$ that belong to $\mathcal{R}_{k-\ell}$ for a fixed $\ell$. In particular, characterize those  $H_{k,k}$  that belong to $\mathcal{R}_{k-1}$. Is the set of crown graphs  $H_{k,k}$  that belong to  $\mathcal{R}_{k-\ell}$ (non-)empty for a fixed $\ell\geq 2$? \end{problem} 

We suspect that each  $H_{k,k}$  belongs to $\mathcal{R}_{k-1}$ for $k\geq 3$. Even if we would manage to prove this statement for a single $k\geq 5$, we would answer at once (negatively) the question in Problem~\ref{problem0}. A step towards solving Problem~\ref{problem2} can be the following observation: Take the concatenation of the permutations representing $H_{k,k}$ discussed above and remove the leftmost $k-1$ letters $1,2,\ldots,k-1$ and the rightmost $k-1$ letters $2',3',\ldots,k'$. The remaining word will still represent $H_{k,k}$, which is not difficult to see. One would then need to remove from the word two more letters, one copy of $k$ and one copy of $1'$, then possibly do some rearrangements of the remaining letters with a hope to obtain a word that would $(k-1)$-word-represent $H_{k,k}$. For example, for $k=3$, we begin with the word representing $H_{3,3}$ in Table~\ref{ex-k-repr-Hkk}, and remove its first two, and its last two letters to obtain the following representation of $H_{3,3}$: $$3'32'1'132'23'1'231'1.$$ Now, remove the leftmost 3 and the leftmost $1'$, make the last letter 1 to be the first letter 1 (that is, make the cyclic one-position shift to the right), and, finally, replace the six rightmost letters in the obtained word by the word obtained by listing these letters in the reverse order to obtain a $2$-word-representation of $H_{3,3}$ (this is the same representation as one would obtain for the cycle $C_6$ if to follow the strategy described in \cite{KL}): $13'2'132'1'321'3'2$.  

Unfortunately, similar steps do not work for $3$-word-representing $H_{4,4}$, and it is not so clear which copies of unwanted letters 4 and $1'$ one should remove before conducting further rearrangements. In either case, even if one could demonstrate how to obtain a $(k-1)$-word-representation of $H_{k,k}$ for some $k\geq 5$, it is not so clear what generic argument would show that no $(k-2)$-word-representation exists, if this would be the case.

\end{document}